\documentclass[a4paper,12pt,reqno]{amsart}
\usepackage[english]{babel}
\usepackage[T1]{fontenc}
\usepackage[left=1.1in,right=1.1in,top=1.2in,bottom=1.2in]{geometry}
\usepackage{times}
\usepackage{microtype}

\usepackage{commath,amssymb,amscd,mathrsfs,mathtools,dsfont,bbm,multirow,array,caption}
\usepackage[all]{xy}
\usepackage{enumerate}
\usepackage{longtable}

\usepackage[usenames,dvipsnames,svgnames,table]{xcolor}
\definecolor{red}{RGB}{255,25,25}
\definecolor{blue}{RGB}{25,50,200}
\usepackage{tikz-cd}
\usetikzlibrary{decorations.pathmorphing}

\usepackage[pagebackref,linktocpage]{hyperref}

\hypersetup{
pdftitle={\@title},			
pdfauthor={\authors},		
colorlinks=true,				
linkcolor=red,				
citecolor=MidnightBlue,		
filecolor=magenta,			
urlcolor=MidnightBlue			
}

\usepackage{cleveref} 

\newtheorem{theorem}{Theorem}[section]
\crefname{theorem}{Theorem}{Theorems}
\newtheorem{lemma}[theorem]{Lemma}
\crefname{lemma}{Lemma}{Lemmas}

\crefname{proposition}{Proposition}{Propositions}
\newtheorem{prop}[theorem]{Proposition}
\crefname{prop}{Proposition}{Propositions}

\crefname{corollary}{Corollary}{Corollaries}

\crefname{cor}{Corollary}{Corollaries}

\crefname{conjecture}{Conjecture}{Conjectures}
\newtheorem{conj}[theorem]{Conjecture}
\crefname{conj}{Conjecture}{Conjectures}
\newtheorem*{conj*}{Conjecture}
\crefname{conj}{Conjecture}{Conjectures}

\theoremstyle{definition}

\crefname{definition}{Definition}{Definitions}
\newtheorem{defn}[theorem]{Definition}
\crefname{defn}{Definition}{Definitions}

\crefname{example}{Example}{Examples}
\newtheorem{notation}[theorem]{Notation}
\crefname{notation}{Notation}{Notation}
\newtheorem*{notation*}{Notation}
\crefname{notation}{Notation}{Notation}
\newtheorem{problem}[theorem]{Problem}
\crefname{problem}{Problem}{Problems}
\newtheorem{question}[theorem]{Question}
\crefname{question}{Question}{Questions}

\crefname{condition}{Condition}{Conditions}

\crefname{assumption}{Assumption}{Assumptions}

\theoremstyle{remark}

\crefname{rmk}{Remark}{Remarks}
\newtheorem*{rmk*}{Remark}
\crefname{rmk}{Remark}{Remarks}

\crefname{remark}{Remark}{Remarks}

\crefname{fact}{Fact}{Facts}
\newtheorem{claim}[theorem]{Claim}
\crefname{claim}{Claim}{Claims}
\newtheorem*{claim*}{Claim}
\crefname{claim}{Claim}{Claims}

\crefname{step}{Step}{Steps}

\crefname{case}{Case}{Cases}

\numberwithin{equation}{section}

\newcommand{\bC}{\mathbf{C}}

\newcommand{\bQ}{\mathbf{Q}}
\newcommand{\bR}{\mathbf{R}}
\newcommand{\bZ}{\mathbf{Z}}

\begin{document}

\title[surface with Picard number 2]{A note on a smooth projective surface with Picard number 2}

\author{Sichen Li}
\address{School of Mathematical Science, Shanghai Key Laboratory of PMMP, East China Normal University, Math. Bldg , No. 500, Dongchuan Road, Shanghai, 200241, P. R. China
\endgraf Department of Mathematics, National University of Singapore, Block S17, 10 Lower Kent Ridge Road, Singapore 119076}
\email{\href{mailto:lisichen123@foxmail.com}{lisichen123@foxmail.com}}
\urladdr{\url{https://www.researchgate.net/profile/Sichen_Li4}}
\begin{abstract}
We characterize the integral Zariski decomposition of a smooth projective surface with Picard number 2 to partially solve  a problem of B. Harbourne, P. Pokora, and H. Tutaj-Gasinska [Electron. Res. Announc. Math. Sci. 22 (2015), 103--108].
\end{abstract}

\subjclass[2010]{
primary 14C20
}


\keywords{integral Zariski decomposition, Picard number 2, K3 surface}

\thanks{The Research was partially supported by the National Natural Science Foundation of China (Grant No. 11471116, 11531007), Science and Technology Commission of Shanghai Municipality (Grant No. 18dz2271000) and the China Scholar Council `High-level university graduate program'.}

\maketitle


\section{Introduction}
In this note we work over the field $\bC$ of complex numbers. By a {\it negative~ curve} on a surface we will always mean a reduced, irreducible curve with negative self-intersection. By a {\it (-k)-curve}, we mean a negative curve $C$ with $C^2=-k<0$.

The bounded negativity conjecture  is one of the most intriguing problems in the theory of projective surfaces and can be formulated as follows.
\begin{conj}\cite[Conjecture 1.1]{B.etc.13}
\label{BNC}
For each smooth complex projective surface $X$ there exists a number $b(X)\ge0$ such that  $C^2\ge-b(X)$ for every negative curve $C\subseteq X$.
\end{conj}
Let us say that a smooth projective surface  $X$ has
\begin{equation*}
b(X)>0
\end{equation*}
 if there is at least one negative curve on $X$.

In \cite{BPS17}, T. Bauer, P. Pokora and D. Schmitz established the following theorem.
\begin{theorem}\cite[Theorem]{BPS17}\label{BPS}
For a smooth projective surface $X$ over an algebraically closed field the following two statements are equivalent:
\begin{enumerate}
  \item $X$ has bounded Zariski denominators.
  \item $X$ satisfies  \cref{BNC}.
\end{enumerate}
\end{theorem}
Let us say that a smooth projective surface  $X$ has
\begin{equation*}
d(X)=1
 \end{equation*}
 if every {\it pseudo-effective divisor} $D$ (cf. \cite[Definition 2.2.25]{Laz04}) on $X$ has an integral Zariski decomposition (cf.  \cref{Fuj-Zar}). An interesting criterion for surfaces to have bounded Zariski denominators was given in \cite{BPS17} as follows.
 \begin{prop}\cite[Proprostion 1.2]{HPT15}\label{1.3}
Let $X$ be a smooth projective surface such that for every curve $C$ one has ${C}^2\ge-1$. Then $d(X)=1$.
\end{prop}
The above proposition introduces a  converse question:
\begin{question} \cite[Question]{HPT15} \label{1.4}
Let $X$ be a smooth projective surface  with $d(X)=1$. Is every negative curve then a (-1)-curve?
\end{question}
  In \cite{HPT15}, the authors  disproved   \cref{1.4} by giving a K3 surface $X$ with $d(X)=1$, Picard number $\rho(X)=2$  and two (-2)-rational curves (cf. \cref{2.12}). However, for a smooth projective surface $X$ with $|\Delta(X)|=1$, sometimes the answer for \cref{1.4} is affirmative, where $\Delta(X)$ is the determinant of the intersection form on the $\mathrm{N\acute{e}ron}$-Severi lattice of $X$. They end by giving the following problem.
\begin{problem}\cite[Problem 2.3]{HPT15}\label{1.5}
Classify all algebraic surfaces with $d(X)=1$.
\end{problem}
To solve  \cref{1.5} partially, for the case when $\rho(X)=2$, we give our main theorem as follows.
\begin{theorem} \label{1.6}
Let $X$ be a smooth projective surface with Picard number 2. If  $b(X)>0$ and $d(X)=1$, then  the following statements hold.
 \begin{enumerate}
\item[(1)] $X$ has at most two negative curves.
 \item[(2)] If $X$ has two negative curves, then $X$ must be one of the following types:  K3 surface,  surface of general type, or one point blow-up of either an abelian surface or a K3 surface with Picard number 1.
\item[(3)] For every negative curve $C$ and every another curve $D$ on $X$, the intersection number, $(C\cdot D)$ is divisible by the self-intersection number $C^2$, i.e. , $C^2|(C\cdot D)$.
\item[(4)]  If the Kodaira dimension $\kappa(X)=-\infty$, then X is a ruled surface with invariant $e=1$ or one point blow up of $\mathbb P^2$.
\item[(5)]  If $\kappa(X)=0$ and the canonical divisor $K_X$ is nef, then  X is a K3 surface admitting an intersection form on the $N\acute{e}ron$-Severi lattice of X  which is
  \[  \begin{pmatrix} a&b\\b&-2\end{pmatrix}\]
   where $a\in\big\{0,-2\big\}$ and $b+a\in 2\bZ_{>0}$.
    \item[(6)]  If $\kappa(X)=1$, then $X$ has exactly one negative curve $C$ and every singular fibre is irreducible. In particular, if every fibre is of type $mI_0$, then the genus $g(C)\ge2$. Here, $mI_0$ is one type in Kodaira's table of singular fibres \textup{(cf. \cite[V.7. Table 3]{BHPV04})}.
 \end{enumerate}
 \end{theorem}
It is well-known that the following SHGH conjecture implies Nagata's conjecture  (cf. \cite[p.772]{Nag59}), which is motivated by Hilbert's 14-th problem.
\begin{conj}\textup{(cf. \cite[Conjectures 1.1, 2.3]{C.etc.13})}\label{1.7}
Let $X$ be  a composite of blow-ups of $\mathbb P^2$ at  points $p_1,\cdots, p_n$ in very general position. Then, every negative curve on $X$ is a (-1)-rational curve.
 \end{conj}
 Finally, we note two corresponding results of   \cref{1.7} as follows.
 \begin{prop}\label{1.8}\textup{(cf. \cite[Theorems 2.2, 2.3]{BPS17})}
Let $X$ be a composite of  blow-ups of $\mathbb P^2$ at n distinct points. Then, $b(X)=1$ if and only if $d(X)=1$.
\end{prop}
Here, a smooth projective surface $X$ has $b(X)=1$ if every negative curve $C$ on $X$ is a (-1)-curve. By  \cref{1.8} and \cref{2.3}, we obtain the following result.
\begin{prop} \label{1.9}
Let $X$ be  a composite of blow-ups of $\mathbb P^2$ at  points $p_1,\cdots, p_n$ in very general position. If there is a negative curve $C$ and another curve $D$ on $X$ such that the intersection matrix of $C$ and $D$ is not negative definite and $C^2\nmid (C\cdot D)$, then  \cref{1.7} fails.
\end{prop}
\section{The proof of  \cref{1.6}}
In this section, we divide our proof of   \cref{1.6} into some steps.
 \begin{notation}\cite[1.6]{Fuj79}
Let $C_1,\cdots, C_q$ be prime divisors. By $V(C_1,\cdots, C_q)$ we denote the $\bQ$-vector space of $\bQ$-divisors generated by $C_1,\cdots, C_q$. $I(C_1,\cdots, C_q)$ denotes the quadratic form on $V(C_1,\cdots, C_q)$ defined by the self-intersection number.
  \end{notation}
\begin{defn}(Fujita-Zariski~decomposition \cite{Zar62, Fuj79}) \label{Fuj-Zar}
Let $X$ be a smooth projective surface and $D$ a pseudo-effective divisor on $X$. Then $D$ can be written uniquely as a sum
 \begin{equation*}
                                                   D=P+N
 \end{equation*}
 of $\bQ$-divisors such that
 \begin{enumerate}
       \item[(1)] $P$ is nef;
        \item[(2)]  $N=\sum_{i=1}^q a_iC_i$ is effective with $I(C_1,\cdots, C_q)$ negative definite if $N\ne0$;
       \item[(3)]  $ P\cdot C_i=0$ for every component $C_i$ of $N$.
 \end{enumerate}
  In particular,  $X$ is said to satisfy $d(X)=1$ if every pseudo-effective divisor $D$ has an integral Zariski decomposition $D=P+N$, i.e. , $P$ and $N$ are integral divisors.
 \end{defn}
\begin{lemma}\label{2.3}
Let X be a smooth projective surface with $b(X)>0$ and $d(X)=1$. Suppose  $I(C_1,C_2)$ is not negative definite. Then, for every negative curve $C_1$ and every another curve $C_2$, $C_1^2|(C_1\cdot C_2)$.
\end{lemma}
\begin{proof}
Let $D(m_1,m_2):=m_1C_1+m_2C_2$ with $m_1,m_2>0$. If $D(m_1,m_2)\cdot C_1<0$ and $D(m_1,m_2)\cdot C_2<0$, then by \cite[Lemma 1.10]{Fuj79}, $I(C_1,C_2)$ is negative definite.
Therefore, $D(m_1, m_2)\cdot C_1<0$ implies that $D(m_1,m_2)\cdot C_2\ge0$.

If $C_1\cdot C_2=0$, then $C_1^2|(C_1\cdot C_2)$, where $C_2^2\ge0$. Hence, we have completed the proof.

Now suppose $C_1\cdot C_2>0$.  Then, there are infinitely many coprime positive integer number pairs $(m_1,m_2)$  such that
 \begin{equation*}
 D(m_1,m_2)\cdot C_1<0, i.e., \frac{m_2}{m_1}< \frac{-C_1^2} {(C_1\cdot C_2)},
  \end{equation*}
 since there are infinitely many prime integers.
 Therefore, we have the following Zariski decomposition:
 \begin{equation*}
             D(m_1,m_2)=m_2(\frac{(C_1\cdot C_2)}{-C_1^2}C_1+C_2)+(m_1-m_2\frac{(C_1\cdot C_2)}{-C_1^2})C_1.
\end{equation*}
 Note that $-C_1^2$ has only finitely many prime divisors, there exists a positive integer $m_2$ such that $(m_2,-C_1^2)=1$. Since $d(X)=1$,  $D(m_1,m_2)$ has an integral Zariski decomposition. Hence,  $C_1^2|(C_1\cdot C_2)$.
 \end{proof}
 By  \cref{2.3}, we can answer the following question in some sense which was posed in \cite{B.etc.13}.
\begin{question}\cite[Question 4.5]{B.etc.13} \label{2.4}
Is there for each $g>1$ a surface $X$ with infinitely many $(-1)$-curves of genus $g$?
\end{question}
\begin{prop}\label{2.5}
 Let $f:X\longrightarrow B$ be a relatively minimal elliptic fibration of a smooth projective surface $X$ with the Kodaira dimension $\kappa(X)=2$  over a smooth base curve $B$ of genus $g\ge2$.  If  $d(X)=1$ and $X$ has infinitely many sections, then $X$ has infinitely many (-1)-curves of genus $g\ge2$ and $q(X)=p_g(X)$. Here, $q(X)$ is the irregularity of $X$, $p_g(X)$ is the geometric genus of $X$.
\end{prop}
\begin{proof}
Since  there exists a section $C$ on $X$,  $X$ has no multiple fibres.
In this case, by the well-known result of Kodaira (cf. \cite[Corollary V.12.3]{BHPV04}), $K_X$ is a sum of a specific choice of $2g(B)-2+\chi(\mathcal O_X)$ fibres of the elliptic fibration.
By \cite[Theorem X.4]{Bea96} and the adjunction formula, $-C^2=\chi(\mathcal O_X)>0$.
If $d(X)=1$, then applying  \cref{2.3} to $C_2=$ a fibre, we obtain $C^2=-1$ and $q(X)=p_g(X)$.
 \end{proof}
 \begin{prop}\label{2.6}
Every smooth projective surface with Picard number 2 satisfies  \cref{BNC}.
\end{prop}
 Indeed,  \cref{2.6} follows from the following claim immediately.
 \begin{claim}\label{2.7}
  If $C_1, C_2$ are two negative curves on a smooth projective surface $X$ with $\rho(X)=2$, then
  \begin{equation*}
  \overline{NE}(X)=\bR_{\ge0}[C_1]+\bR_{\ge0}[C_2]
  \end{equation*}
and  $C_i ~(i=1,2)$  are the only two negative curves.
 \end{claim}
 \begin{proof}
By \cite[Lemma 1.22]{KM98}, $C_1, C_2$ are both extremal curves in the closed Mori cone $\overline{NE}(X)$ which has only two extremal rays since $\rho(X)=2$. Thus, the first part of  \cref{2.7} follows.  Moreover, if $C_3$ is another negative curve (except for $C_1, C_2$), then the class $[C_3]$ is also extremal. Since $\rho(X)=2$,  $C_3\equiv a_iC_i$  for $i=1$ or 2 with $a_i\in\bQ_+$. Thus, $0\le C_i\cdot C_3=a_iC_i^2<0$, a contradiction.
 \end{proof}
 By \cref{2.3}, for the case when $\rho(X)=2$, we have the following result.
 \begin{claim}\label{2.8}
 Let $X$ be  a smooth projective surface with $\rho(X)=2$.
 If $b(X)>0$ and $d(X)=1$, then for every negative curve $C$ and every another curve $D$ on $X$, $C^2|(C\cdot D)$.
 \end{claim}
It is well-known that the smooth projective surfaces satisfy the minimal model conjecture (cf. \cite{KM98, BCHM10}) as follows.
 \begin{lemma}\label{abund}
Let $X$ be a smooth projective surface. If the canonical divisor $K_X$ is pseudo-effective, then the Kodaira dimension $\kappa(X)\ge0$.
\end{lemma}
 \begin{claim}\label{2.10}
Let $X$ be  a smooth projective surface with $\rho(X)=2$. If $\kappa(X)=-\infty, b(X)>0$ and $d(X)=1$, then X is a ruled surface with invariant $e=1$ or one point blow-up of $\mathbb P^2$.
\end{claim}
\begin{proof}
 Let $S$ be a relatively minimal model of $X$.
 A smooth projective surface $S$ is relatively minimal if it has no  (-1)-rational curves.
 By the classification of relatively minimal surfaces (cf. \cite{Har77, BHPV04,KM98}), it must be one of the following cases: a surface with nef canonical divisor, a ruled surface or $\mathbb P^2$.
  Since $\kappa(X)=-\infty$, by  \cref{abund}, $K_S$ is not nef.
  Therefore, $S$ is either a ruled surface or $\mathbb P^2$.
 As a result, $\rho(X)=2$ implies that $X$ is either a ruled surface or one point blow-up of $\mathbb P^2$.

Now suppose $X$ is ruled.  Let $\pi: X\longrightarrow C$ be a ruled surface over a curve $C$ with invariant $e$, let $C_0\subseteq X$ be a suitable section, and let $f$ be a fibre.  Then, we have the following ( cf. \cite[Propositions V.2.3 and V.2.9]{Har77}):
 \begin{equation*}
                  \mathrm{Pic}\ X\simeq\bZ C_0\oplus\pi^{*}\mathrm{Pic}\ C, C_0\cdot f=1, f^2=0, C_0^2=-e.
\end{equation*}
Let $D=aC_0+bf$ be a curve on  $X$. By \cite[Proposition V.2.20]{Har77}, $D^2<0$ if and only if $D=C_0$ and $e>0$. Since $d(X)=1$, applying  \cref{2.8} to a fibre $f$, we obtain $e=1$.
 \end{proof}
 \begin{claim} \label{2.11}
 Let $X$ be a smooth projective surface with $\rho(X)=2$.
 If $X$ has two negative curves, then $X$ must be one of the following types: K3 surface, surface of general type, or one point blow-up of either an abelian surface or a K3 surface with Picard number 1.
 \end{claim}
 \begin{proof}
Suppose $X$ has two negative curves $C_1,C_2$.
By  \cref{2.10}, if $\kappa(X)=-\infty$, then $X$ has at most one negative curve.
Thus, $\kappa(X)\ge0$, i.e., there exists a positive integral number $m$ such that $h^0(X, \mathcal O_X(mK_X))\ge0$.
 Therefore, $K_X$ is a $\mathbb Q$-effective divisor.  As a result, by \cref{2.7}, we have the following result:
 \begin{equation*}
K_X\in\overline{NE}(X)=\bR_{\ge0}[C_1]+\bR_{\ge0}[C_2], ~i.e., K_X\equiv a_1C_1+a_2C_2, a_1,a_2\ge0.
 \end{equation*}
Hence, we have three cases as follows.
\begin{enumerate}
\item[(1)] $a_1,a_2>0$. Then $K_X$ is an  interior point of $\overline{NE}(X)$, and by \cite[Lemma 10.5]{Iit82} or  \cite[Theorem 2.2.26]{Laz04}, $K_X$ is big. Thus, $X$ is a surface of general type.
\item[(2)]  $a_1=a_2=0$. Then $K_X\equiv0$, i.e., $X$ is minimal. By Enriques Kodaira classification (cf.  \cite[Theorem V.6.3]{Har77}), $X$ has the following cases: K3 surface, Enriques surface, abelian surface, hyperelliptic surface, where in the latter two cases, $X$ has no any rational curves.
    By the genus formula, every negative curve $C$ on $X$ is a (-2)-rational curve.
    As a result, $X$ is either a K3 surface or an Enrique surface. Moreover, since an Enriques surface $X$ has $\rho(X)=10$ by \cite[Proposition VIII.15.2]{BHPV04}, $X$ is a K3 surface.
\item[(3)]  $a_1>0, a_2=0$.  Then $K_X\equiv a_1C_1$.  Since $K_X$ is a $\mathbb Q$-effective divisor, there exists an effective divisor $D$ such that $K_X\sim_{\bQ} D$. Therefore, we can find an effective divisor $D'\neq C_1$ such that
\begin{equation*}
                   a_1'C_1+D'=D\equiv a_1C_1,
\end{equation*}
where $a_1'\ge0$, and $D'$ and $C_1$ have no common components.  Then, $D'\equiv (a_1-a_1')C_1$.

If $a_1=a_1'$,then $K_X\sim_{\bQ} a_1C_1$ with $a_1>0$. By the genus formula, $C_1$ is a (-1)-rational curve.   In this case, $\kappa(X)=\kappa(X,C)=0$. By Castelnuovo's contractibility criterion (cf. \cite[Theorem V.5.7]{Har77} or \cite[Thereom II.17]{Bea96}), $X$ is a one point  blow-up  of either an abelian surface or a  K3 surface with Picard number 1.

If $a_1>a_1'$, then $D'\cdot C_1=(a_1-a_1')C_1^2<0$, a contradiction.

If $a_1<a_1'$, on the one hand $D'+(a_1'-a_1)C_1\equiv0$ with $a_1'-a_1\ge0$; on the other hand, there is an ample divisor $H$ on $X$ such that $(D'+(a_1'-a_1)C_1)\cdot H=0$. Since the restriction of an ample divisor to a curve is still ample,  $D'+(a_1'-a_1)C_1=0$, i.e., $D=a_1C_1,a_1=a_1'$, a contradiction.
\end{enumerate}
  \end{proof}
Theorem A of \cite{HPT15} is a special case of the following  \cref{2.12}.
 \begin{claim} \label{2.12}
Let $X$ be  a smooth projective surface with $\rho(X)=2$. If $\kappa(X)=0,b(X)>0, d(X)=1$ and $K_X$ is nef, then  X is a K3 surface admitting the intersection form on the $N\acute{e}ron$-Severi lattice of X, which is
  \[  \begin{pmatrix} a&b\\b&-2\end{pmatrix}\]
where $a\in\big\{0,-2\big\}$ and $b+a\in 2\bZ_{>0}$.
\end{claim}
\begin{proof}
Since $\kappa(X)=0$ and $K_X$ is nef,  $K_X\equiv 0$. By the genus formula, every negative curve on $X$ is a (-2)-rational curve. Note that abelian surfaces and hyperelliptic surfaces have no rational curves. Then by \cite[Theorem V.6.3]{Har77} and \cite[Proposition VIII.15.2]{BHPV04},  we know that $X$ is a K3 surface.  In \cite{Kov94}, the author showed that $\overline{NE}(X)=\mathbb R_{\ge0}[C_1]+\mathbb R_{\ge0}[C_2]$, where either $C_1^2=C_2^2=-2$ or $C_1^2=0$ and $C_2^2=-2$. Since $d(X)=1$, applying  \cref{2.8} to a negative curve $C_i$, $C_i^2|(C_1\cdot C_2)$. By the Hodge index theorem, $(C_1\cdot C_2)^2-C_1^2\cdot C_2^2>0$. Finally, the desired result holds by using \cite[Corollary 1.4]{Kov94}.
 \end{proof}
The following lemma is well known.
 \begin{lemma}\cite[Proposition III.11.4]{BHPV04} \label{2.13}
Let $p: X\longrightarrow B$ be an elliptic fibration from a smooth projective surface $X$ to  a curve $B$. If every fibre is of type $mI_0$, then $c_2(X)=0$.
\end{lemma}
 \begin{claim}\label{2.14}
 Let $X$ be  a smooth projective surface with $\rho(X)=2$. If $\kappa(X)=1$ and $b(X)>0$, then $X$ has exactly one negative curve $C$ and every singular fibre is irreducible. In particular, if every fibre is of type $mI_0$, then $g(C)\ge2$.
 \end{claim}
\begin{proof}
Since $\kappa(X)=1, \rho(X)=2$ and $\kappa(X)$ is a birational invariant, $K_X$ is nef.
By \cite[Proposition IX.2]{Bea96},  we have $K_X^2=0$ and there is a surjective morphism $p: X\longrightarrow B$ over a smooth curve $B$, whose  general fibre  $F$ is an elliptic curve. Suppose $F=\sum_{i=1}^r m_iC_i$ with $m_i\in\bZ_{>0},r\ge2$ is a singular fibre.
Then by Zariski's Lemma (cf. \cite[Lemma III.8.2]{BHPV04}),
\begin{equation*}
(F-m_1C_1)^2<0,C_1^2<0.
\end{equation*}
Therefore, $X$ has at least two negative curves, a contradiction (cf. \cref{2.10}).
As a result, every singular fibre is irreducible and $X$ has exactly one negative curve $C$ since $b(X)>0$.
Moreover, if every fibre is of type $mI_0$, then by  \cref{2.13}, we have $c_2(X)=0$.
Hence, by \cite[Theorem 2.4]{B.etc.13}, we have the following inequality:
\begin{equation*}
                              0<-C^2\le2g(C)-2.
 \end{equation*}
 Thus, $g(C)\ge2$.
 \end{proof}
 \begin{proof}[$Proof~of ~\cref{1.6}$]
By  \cref{2.7,2.8,2.10,2.11,2.12,2.14}, we have completed the proof of  \cref{1.6}.
\end{proof}
We end by asking the following two questions.
 \begin{question}
Is there a positive constant $l$ such that $b(X)\le l$ for any smooth projective surface $X$ with $\rho(X)=2$  and $d(X)=1$ ?
\end{question}
\begin{question}
Let $X$ be a smooth projective surface with Picard number $\rho(X)\ge3$ and $d(X)=1$. Take some negative curves $C_1,\cdots, C_k $ with $k\ge2$ on $X$ such that $I(C_1,\cdots, C_k)$  is negative definite.  Is the determinant $\mathrm{det}(C_i\cdot C_j)_{1\le i,j\le k}$  equal to $(-1)^k$ ?
\end{question}
\section*{Acknowledgments}
The author gratefully acknowledge Prof. De-Qi Zhang  for his useful comments and crucial suggestions. He also thank Prof. Rong Du and Prof. Piotr Pokora for helpful conversations and the referee for several suggestions.



\bibliographystyle{amsalpha}
\bibliography{../mybib}

\end{document}